\documentclass[a4paper, 11pt]{amsart}
\usepackage[T1]{fontenc}
\usepackage{amssymb}
\usepackage[utf8]{inputenc}
\usepackage{hyperref}
\usepackage{bm}
\usepackage{comment}
\usepackage{xcolor}

\newcommand{\Z}{{\ensuremath{\mathbb{Z}}}}

\newcommand{\spa}{\mbox{span}}
\newcommand{\supp}{\mbox{supp}}

\makeatletter
\@namedef{subjclassname@2020}{%
  \textup{2020} Mathematics Subject Classification}
\makeatother

\allowdisplaybreaks[2]

\newtheorem{theorem}{Theorem}[section]
\newtheorem{proposition}[theorem]{Proposition}
\newtheorem{lemma}[theorem]{Lemma}
\newtheorem{corollary}[theorem]{Corollary}
\theoremstyle{definition}
\newtheorem{remark}[theorem]{Remark}
\newtheorem{definition}[theorem]{Definition}

\newtheorem{question}[theorem]{Question}

\numberwithin{equation}{section}

\begin{document}

\title[Values of graded $\ast-$polynomials]{Values of  multilinear graded $\ast$-polynomials on upper triangular matrices of small dimension}

\author{Pedro Fagundes}
\address{Department of Mathematics, State University of Campinas, 651 S\'ergio Buarque de Holanda, 13083-859 Campinas, SP, Brazil}
\email{pedro.fagundes@ime.unicamp.br}

\thanks{P.\ Fagundes was supported by Grant 2019/16994-1 and Grant 2022/05256-2, São Paulo Research Foundation (FAPESP)}



\begin{abstract}
Let $F$ be an algebraically closed field of characteristic different from $2$. We show that the images of multilinear $*$-graded polynomials on $UT_2$ are homogeneous vector spaces. An analogous result holds for $UT_3$ endowed with non-trivial grading. We further show that these results are optimal, in the following sense: there exist multilinear $*$-graded polynomials whose image on $UT_n$ ($n\geq 3$) with the trivial grading is not a vector space, and whose image on $UT_n$ ($n\geq4$) with the natural $\mathbb{Z}_{n}$-grading is also not a vector space. In particular, an analog of the L'vov-Kaplansky conjecture can not be expected in the setting of algebras with (graded) involutions.
\end{abstract}

\maketitle

\section{Introduction}

Let $F$ be a field and $\mathcal{A}$ be an associative algebra over $F$. Denoting by $F\langle X \rangle$ the free associative algebra, we define the image of a polynomial $f(x_1,\dots,x_m)\in F\langle X \rangle$ on $\mathcal{A}$ as the image of the function induced by $f$ on $\mathcal{A}$, that is
\[
\begin{array}{cccc}
\tilde{f}\colon & \mathcal{A}^m & \rightarrow & \mathcal{A} \\
 & (a_1,\dots,a_m) & \mapsto & f(a_1,\dots,a_m).
\end{array}
\]
By a multilinear polynomial $f(x_1,\dots,x_m)\in F\langle X \rangle$ we mean a polynomial such that each variable $x_i$ appears in every monomial of $f$ exactly once. Hence  a multilinear polynomial has the following form
\[
f(x_1,\dots,x_m)=\sum_{\alpha\in S_m}\alpha_{\sigma}x_{\sigma(1)}\cdots x_{\sigma(m)},
\]  
where $\alpha_{\sigma}\in F$ and $S_m$ stands for the symmetric group of permutations of $\{1,\dots,m\}$.

A challenging problem concerning images of polynomials on algebras is the L'vov-Kaplansky conjecture, which states that the image of a multilinear polynomial on the full matrix algebra $M_n(F)$ is always a vector space. An equivalent form of stating this conjecture is saying that such image is one of the following four subspaces: $\{0\}, F$ (the space of scalar matrices), $sl_n(F)$ (the space of traceless matrices) or $M_n(F)$. Of course each one of these four subspaces can be realized as images of multilinear polynomials on the full matrix algebra. For instance, $\{0\}$ and $F$ is always the image of multilinear polynomial identities and central polynomials, respectively. The space $sl_n(F)$ is the image of the commutator $[x_1,x_2]=x_1x_2-x_2x_1$ (see \cite{Albert,Shoda}), and $M_n(F)$ is always the image of multilinear polynomials $f$ such that $f(1,\dots,1)\neq0$. Positive results on the L'vov-Kaplansky conjecture are only known for polynomials of degree $2$ (which are an easy consequence from \cite{Albert,Shoda}) and for $2\times 2$ matrices over quadratically closed fields (see \cite{Belov1}). Partial results are also know for polynomials of degree $3$ and $3\times 3$ matrices, however the conjecture remains open on these cases as well. The interested reader might consider the survey paper \cite{Belov2} for a compilation of results concerning the L'vov-Kaplansky conjecture.

In recent years, variations of the aforementioned conjecture have been extensively studied. Let us denote by $UT_n(F)$ (or just $UT_n$ if the field is clear from the context) the algebra of $n\times n$ upper triangular matrices and by $J$ its Jacobson radical. In 2022, Gargate and de Mello \cite{Gargate} proved that the image of a multilinear polynomial on $UT_n$ is always $UT_n$ or some power of $J$, provided the ground field is infinite. Some additional structure can be also considered in the algebra and the question whether the image of a multilinear polynomial (with respect to the additional structure) on the algebra is a vector space can also be posed. For instance, images of multilinear polynomials on $UT_n$ endowed with involutions and gradings were recently explored. More precisely, in 2022 Franca and Urure \cite{Franca1} considered $UT_n$ ($n\leq 4$) endowed with the reflexive involution and proved that the image of a multilinear Lie polynomial on the skew-symmetric part of $UT_n$ is always a vector space. The same authors also studied the Jordan case for polynomials of degree $\leq3$ on the symmetric part of $UT_n$ endowed with both reflexive and sympletic involutions (see \cite{Franca2}). Concerning the graded setting, the author and Koshlukov (see \cite{FagundesKoshlukov}) classified the images of multilinear polynomials on some $\mathbb{Z}_q$-gradings on $UT_n$ (including the $\mathbb{Z}_n$ natural one). 

Besides the fact of these variations been interesting problems by themselves, connections between them and the L'vov-Kaplansky conjecture also appear. For instance we mention Theorem 5.2 from \cite{FagundesKoshlukov} and the paper \cite{Gargate2} . 

In this paper we are interested in studying the images of (associative) multilinear polynomials on the algebra $UT_n$ endowed with some graded involution. In light of the L'vov-Kaplansky conjecture and its variants, the following question is therefore naturally posed.

\begin{question}\label{question}
Let $UT_n$ be endowed with some (graded) involution and let $f$ be a multilinear (graded) $*$-polynomial. Is the image of $f$ on $UT_n$ a vector space?
\end{question}

As  we shall see, the answer to Question \ref{question} is negative in general. The affirmative answers take place only in small cases, that is, $2\times 2$ matrices, and $3\times 3$ matrices with a nontrivial grading. Beyond the cases mentioned above, we are able to find multilinear polynomials such that their image on $UT_n$ is not a vector space.

The paper is organized as follows: in Section 2 we recall the main definitions, results and notations that will be used throughout the text. In Section 3 we prove that the image of multilinear polynomials on $UT_2$ with some graded involution is always a homogeneous vector space. In section 4 an analogous result is proved for $UT_3$ provided that the grading on this algebra is not trivial. Still in this section we present examples of multilinear polynomials of degree $2$ such that their images on $UT_n$ are not vector spaces if $n\geq3$ and $UT_n$ is endowed with the trivial grading, and $n\geq4$ and $UT_n$ has the natural $\mathbb{Z}_n$-grading.

\section{Preliminaries}

Let $\mathcal{A}$ be an associative algebra over a field $F$ of characteristic different from $2$. An involution on $\mathcal{A}$ is defined as an additive map $*:\mathcal{A}\rightarrow\mathcal{A}$ such that $(ab)^{*}=b^{*}a^{*}$ and $(a^{*})^{*}=a$, for all $a,b\in\mathcal{A}$, where $a^{*}:=*(a)$. If $\lambda^{*}=\lambda$ for all $\lambda\in F$, we say that $*$ is an involution of the first kind, and of the second kind otherwise. Note that involutions of the first kind are actually linear maps. We will only be interest in involutions of the first kind. Given two algebras $\mathcal{A}$ and $\mathcal{B}$ with involutions $*_1$ and $*_2$, respectively, we say that $\varphi:\mathcal{A}\rightarrow\mathcal{B}$ is a homomorphism of algebras with involution if $\varphi$ is an algebra homomorphism satisfying $\varphi(a^{*_1})=\varphi(a)^{*_2}$, for all $a\in \mathcal{A}$. In case $\varphi$ is bijective we say that $\varphi$ is an isomorphism of algebras with involution. Two involutions $*_1$ and $*_2$ on an algebra $\mathcal{A}$ are said to be equivalents if $(\mathcal{A},*_1)$ and $(\mathcal{A},*_2)$ are isomorphic as algebras with involution. An element $a\in \mathcal{A}$ is called symmetric if $a^{*}=a$, and skew-symmetric if $a^{*}=-a$. We denote by $\mathcal{S}$ and $\mathcal{K}$ the subspaces of the symmetric and skew-symmetric elements from $\mathcal{A}$. Note that since the characteristic of $F$ is different from $2$, then $\mathcal{A}=\mathcal{S}\oplus \mathcal{K}$. 

Involutions on the upper triangular matrix algebra $UT_{n}$ are well understood. We recall two particular types of involutions first.
\begin{itemize}
\item[1.] (Reflexive) $r:UT_{n}\rightarrow UT_{n}$ given by $A^{r}=JA^{t}J$, $A\in UT_{n}$, where $A^{t}$ stands for the usual transposition of matrices and \linebreak $J=e_{1n}+e_{2,n-1}+\cdots+e_{n1}$.
\item[2.] (Symplectic) $s:UT_{2m}\rightarrow UT_{2m}$ given by $A^{s}=DA^{r}D^{-1}$, $A\in UT_{n}$, where $D=e_{11}+\cdots+e_{mm}-e_{m+1,m+1}-\cdots -e_{2m,2m}$.
\end{itemize}

\begin{theorem}[\cite{VincenzoKoshlukovScala}]\label{classification1}
    Every involution on $UT_n$ is equivalent to either the reflexive or the symplectic one.
\end{theorem}

An involution $*$ on $\mathcal{A}$ is called a $G$-graded involution for some group $G$ (or just graded involution), if $*$ leaves invariant the homogeneous components of some $G$-grading $\Gamma$ on $\mathcal{A}$. This means that writing 
\[
\Gamma: \mathcal{A}=\bigoplus_{g\in G}\mathcal{A}_{g}
\]
as a direct sum of vector subspaces such that $\mathcal{A}_{g}\mathcal{A}_{h}\subset \mathcal{A}_{gh}$, for all $g$, $h\in G$, then 
\[
\mathcal{A}_{g}^{*}\subset \mathcal{A}_{g},
\]
for all $g\in G$. 

The subspaces $\mathcal{A}_g$ from the grading $\Gamma$ are called homogeneous components. We also define the support of the grading $\Gamma$ as the subset $\supp(\Gamma)=\{g\in G\mid \mathcal{A}_g \neq 0\}$.

A subspace $\mathcal{V}\subset \mathcal{A}$ is said to be homogeneous (with respect to the $G$-grading on $\mathcal{A}$) if $\mathcal{V}$ inherits the grading from $\mathcal{A}$, that is,
\[
\mathcal{V}=\bigoplus_{g\in G}(\mathcal{V}\cap \mathcal{A}_g).
\]

In particular, since the homogeneous components of $\mathcal{A}$ are invariant under the involution $*$, one can easily check that the symmetric and skew-symmetric parts of $\mathcal{A}$ are homogeneous subspaces. For $g\in G$, denoting $\mathcal{K}_g=\mathcal{K}\cap \mathcal{A}_g$ and $\mathcal{S}_g=\mathcal{S}\cap \mathcal{A}_g$ we therefore have
\[
\mathcal{A}=\bigoplus_{g\in G}(\mathcal{K}_g\oplus \mathcal{S}_g).
\]

Given two algebras $\mathcal{A}$ and $\mathcal{B}$ with $G$-graded involutions $*_1$ and $*_2$, respectively, we say that $\varphi:\mathcal{A}\rightarrow\mathcal{B}$ is a homomorphism of algebras with graded involution if $\varphi$ is an algebra homomorphism such that $\varphi(\mathcal{A}_{g})\subset \mathcal{B}_{g}$, for all $g\in G$, and $\varphi(a^{*_1})=\varphi(a)^{*_2}$, for all $a\in \mathcal{A}$. In case $\varphi$ is bijective we say that $\mathcal{A}$ and $\mathcal{B}$ are isomorphic as algebras with graded involution.

Graded involutions on $UT_{n}$ are also well known, but before introducing them let us recall an important example of gradings on $UT_n$. A $G$-grading on $UT_{n}$ is called elementary if the matrix units $e_{ij}$ ($1$ in the entry $(i,j)$ and zero elsewhere) are homogeneous elements in the grading, or equivalently, it is that given by a sequence $(g_{1},\dots,g_{n})\in G^{n}$ such that $\deg(e_{ij})=g_{i}^{-1}g_{j}$.  The following classification was given in \cite{FidelisGoncalvesDinizYasumura}.

\begin{theorem}\label{classificationgradedinvolution}
Let $F$ be an algebraically closed field of characteristic different from $2$, and let $G$ be a group. Let $\Gamma_{1}$ be a $G$-grading on $UT_{n}(F)$ such that $supp(\Gamma_{1})$ generates $G$. Let $\varphi_{1}$ be a graded involution on $\Gamma_{1}$. Then $(\Gamma_{1},\varphi_{1})$ is isomorphic to $(\Gamma_{2},\varphi_{2})$ where $\Gamma_{2}$ is the elementary grading on $UT_{n}$ induced by the sequence $(g_{1},\dots,g_{n})\in G^{n}$ such that $g_{1}g_{n}=g_{2}g_{n-1}=\cdots=g_{n}g_{1}$ and $\varphi_{2}$ is either $r$ or $s$.
\end{theorem}

In dealing with polynomial identities and images of polynomials on algebras with graded involution, one needs to consider the appropriate polynomials instead of the ordinary ones from the free associative algebra. For that reason, let us introduce the free algebra in the class of algebras with graded involution. Let $G$ be a group and let $X_G=\{x_{i,g},x_{i,g}^{*}|g\in G,i=1,2,\dots\}$ be a set of non-commuting variables. Consider the free associative algebra $F\langle X_G \rangle$ generated by $X$. It turns out that $F\langle X \rangle$ is $G$-graded with involution (which we also denote by $*$). This algebra is then called free $G$-graded associative algebra with involution and will be denoted by $F\langle X_G|*\rangle$. The elements from $F\langle X_G |* \rangle$ will be called graded $*$-polynomials.

Writing $y_{i,g}=(x_{i,g}+x_{i,g}^{*})/2$ and $z_{i,g}=(x_{i,g}-x_{i,g}^{*})/2$, and considering the sets $Y_G=\{y_{i,g}\mid g\in G, i=1,2,\dots,\}$ and $Z_G=\{z_{i,g}\mid g\in G,i=1,2,\dots\}$, we have then that $F\langle X_G|*\rangle$ is actually $F\langle Y_G\cup Z_G \rangle$. In case $G$ is the trivial group, then we use the notations $F\langle X|* \rangle$ and $F\langle Y\cup Z \rangle$ instead of the ones introduced before.

\begin{definition}
    Let $f=f(y_{1,g_1},\dots,y_{l,g_l},z_{l+1,g_{l+1}},\dots,z_{m,g_m})\in F\langle Y_G\cup Z_G \rangle$ and let $\mathcal{A}$ be a $G$-graded algebra with graded involution. We define the image of $f$ on $\mathcal{A}$ (denoted by $f(\mathcal{A})$) as the image of the polynomial function
    \[
\begin{array}{cccc}
\tilde{f}\colon & \mathcal{S}_{g_1}\times\cdots\mathcal{S}_{g_l}\times \mathcal{K}_{g_{l+1}}\times\cdots\times\mathcal{K}_{g_{m}} & \rightarrow & \mathcal{A} \\
 & (a_1,\dots,a_l,b_{l+1},\dots,b_m) & \mapsto & f(a_1,\dots,a_l,b_{l+1},\dots,b_m)
\end{array},
\]
where $a_{g_i}\in\mathcal{S}_{g_i}$ and $b_{g_j}\in \mathcal{K}_{g_j}$.
\end{definition}

In case $f(\mathcal{A})=0$ we say that $f=0$ is a graded $*$-polynomial identity for the algebra $\mathcal{A}$. The set of all graded $*$-polynomial identities for an algebra $\mathcal{A}$ will be denoted by $Id_{(G,*)}(\mathcal{A})$ (in case $\mathcal{A}$ is endowed only with an involution $*$ we denote the set of its $*$-identities by $Id(\mathcal{A},*)$). Notice that $Id_{(G,*)}(\mathcal{A})$ is invariant under all endomorphisms of $F\langle Y_G\cup Z_G \rangle$ that respect the involution $*$.  

The space of multilinear graded polynomials with involution is given by 
\[
P_m^{G}={\rm span}\{\xi_{\sigma(1)}\cdots \xi_{\sigma(m)}|\xi_i\in\{y_{i,g_i},z_{i,g_i}\}, i=1,\dots,m\}.
\]

We notice here that polynomial functions given by the polynomials in $P_m^G$ are not multilinear functions. For instance, the polynomial $f=y_{1,g} z_{2,h}+z_{1,g} z_{2,h}$ is a multilinear graded $*$-polynomial however the function given by it on some graded algebra with involution is not bilinear. In order to obtain multilinear functions we will need to consider the following multilinear graded $*$-polynomials from $P_m^G$.

Let us consider the subspace $P_{m,l}^G$ of $P_m^G$ given by 
\[
P_{m,l}^G={\rm span}\{\xi_{\sigma(1)}\cdots \xi_{\sigma(m)}|\xi_i=y_{i,g_i}, i=1,\dots,l, \xi_i=z_{i,g_i}, i=l+1,\dots,m\}.
\] 
The functions induced by the polynomials from $P_{m,l}^{G}$ are clearly multilinear ones.

Hence Question \ref{question} can be rephrased as follows.

\begin{question}\label{question2}
    Let $UT_n$ be endowed with some graded involution and let $f\in P_{m,l}^{G}$. Is $f(UT_n)$ a vector space?
\end{question}


The next proposition takes its place when one is attempting to give a positive solution to Question \ref{question2} regardless the graded involution.

\begin{proposition}\label{epimorphism}
Let $f\in F\langle X_G|* \rangle$ and let $\mathcal{A}$ and $\mathcal{B}$ be $G$-graded algebras with involutions $*_1$ and $*_2$, respectively. Let $\varphi:\mathcal{A}\rightarrow \mathcal{B}$ be an epimorphism of algebras with graded involution and assume $f(\mathcal{A})$ is a homogeneous subspace of $\mathcal{A}$. Then $f(\mathcal{B})$ is also a homogeneous subspace of $\mathcal{B}$. 
\end{proposition}

\begin{proof}
Recalling that $\varphi(a^{*_1})=\varphi(a)^{*_2}, a\in\mathcal{A}$, the proof follows as in \cite[Proposition 6.1]{FagundesKoshlukov} with some mild and obvious adjustments. 
\end{proof}
The following was proved in \cite[Lemma 3]{Malev}.

\begin{lemma}
Let $F$ be a field, and let $\mathcal{V}_{1},\dots,\mathcal{V}_{m},\mathcal{V}$ be vector spaces over $F$. Assume that the image of a multilinear map $f:\prod_{i=1}^{m}\mathcal{V}_{i}\rightarrow \mathcal{V}$ contains two linearly independent vectors. Then the image of $f$ contains a $2$-dimensional vector subspace. 
\end{lemma}

As a consequence we have the following corollary.

\begin{corollary}\label{corollarydimension2}
Assume that the image of a multilinear polynomial $f$ on some algebra $\mathcal{A}$ is contained in some $2$-dimensional vector space. Then, $f(\mathcal{A})$ is a vector subspace. 
\end{corollary}

\begin{proof}
Let $f(\mathcal{A})\subset \mathcal{V}$, where $\dim(\mathcal{V})=2$. If $f(\mathcal{A})$ contains two linearly independent vectors, by the previous lemma $f(\mathcal{A})=\mathcal{V}$. Now assume that any two vectors in $f(\mathcal{A})$ are linearly dependent.  Then $f(\mathcal{A})\subset \mathcal{U}$, where $\mathcal{U}$ is a $1$-dimensional subspace. Hence, $f(\mathcal{A})=\{0\}$ or $f(\mathcal{A})=\mathcal{U}$.
\end{proof}

\section{$2\times 2$ matrices}

The main goal of this section is to prove the following theorem.

\begin{theorem}\label{ut2}
    Let $F$ be an algebraically closed field of characteristic different from $2$ and let $f\in P_{m,l}^G$. Assume that $UT_{2}$ is endowed with some $G$-graded involution $*$ on a grading $\Gamma$ such that $supp(\Gamma)$ generates $G$. Then the image of $f$ on $UT_{2}$ is a homogeneous vector space.   
\end{theorem}


Let us recall that the reflexive and the symplectic involutions are given on $UT_{2}$ respectively by:
\[
\begin{pmatrix}
a&b\\
0&c
\end{pmatrix}^{r}=\begin{pmatrix}
    b&c\\
    0&a
\end{pmatrix} \ \mbox{and} \ \begin{pmatrix}
    a&b\\
    0&c
\end{pmatrix}^{s}=\begin{pmatrix}
    b&-c\\
    0&a
\end{pmatrix}
\]

In light of Theorem \ref{classificationgradedinvolution}, one can easily see that a graded involution on $UT_{2}$ is isomorphic to either the reflexive or sympletic one where $UT_2$ is endowed with either the
\begin{itemize}
\item $\Gamma_{2,1}$: the trivial grading;
\item $\Gamma_{2,2}$: $UT_2=\mathcal{A}_1 \oplus \mathcal{A}_g$, where $\mathcal{A}_1=\spa\{e_{11},e_{22}\}$ and $\mathcal{A}_g=\spa\{e_{12}\}$.
\end{itemize}

\subsection{Grading \boldmath{$\Gamma_{2,1}$}}

We shall deal with the reflexive and symplectic involutions in two different subsections.

\subsubsection{Reflexive involution}

In this subsection we denote $\mathcal{A}=UT_2$ equipped with the reflexive involution, and $\mathcal{S}$ and $\mathcal{K}$ will stand for the symmetric and skew-symmetric elements from $\mathcal{A}$. One can easily check that 
\[
\mathcal{S}=\spa\{e_{11}+e_{22},e_{12}\} \ \mbox{and} \ \mathcal{K}=\spa\{e_{11}-e_{22}\}. 
\]
We also denote by $J$ the linear span of $e_{12}$.

We recall now the following proposition from \cite{VincenzoKoshlukovScala} (see also \cite{UrureGoncalves}).

\begin{proposition}\label{pgeneratorsreflection}
Let $F$ be a field of characteristic different from $2$. Then, $F\langle Y \cup Z\rangle$ is generated, modulo $Id(UT_{2},r)$, by the following polynomials
\begin{eqnarray}\nonumber 
y_{1}^{p_{1}}\cdots y_{n}^{p_{n}}z_{1}^{q_{1}}\cdots z_{m}^{q_{m}}[z_{m},y_{k}] \ \mbox{and} \ y_{1}^{p_{1}}\cdots y_{n}^{p_{n}}z_{1}^{q_{1}}\cdots z_{m}^{q_{m}}
\end{eqnarray}
where $n\geq 1,m\geq 1,p_{1},\dots,p_{n},q_{1},\dots,q_{m}\geq 0,k\geq 1$.
\end{proposition}

We consider a set of commuting variables $W=\{w_j^{(i)}|i,j=1,2,\dots\}$ and the commutative polynomial algebra $F[W]$. Let us also set the following evaluations of the symmetric and skew-symmetric variables by matrices with entries in $F[W]$:
\begin{eqnarray}\label{evaluation1}
y_{i}=\begin{pmatrix}
     w_{1}^{(i)}&w_{2}^{(i)}  \\
     & w_{1}^{(i)}
\end{pmatrix}
\ \mbox{and} \  z_{j}=\begin{pmatrix}
     w_{1}^{(j)}&0  \\
     & -w_{1}^{(j)}
\end{pmatrix}.
\end{eqnarray}

\begin{remark}
In a product $x_{1}\cdots \widehat{x_{i}} \cdots x_{m}$, the hat \ $\widehat{ }$ \  means that the variable $x_{i}$ is missing.
\end{remark}

\begin{lemma}\label{lemma12entry}
In light of \eqref{evaluation1}, the entry $(1,2)$ of $y_{1}\cdots y_{l}$ is given by 
\[
\sum_{i=1}^{m}w_{1}^{(1)}\cdots \widehat{{w}_{1}^{(i)}}\cdots w_{1}^{(l)}w_{2}^{(i)}
\]
\end{lemma}

\begin{proof}
It is enough to apply induction on $l$.
\end{proof}

In the next result we identify the space of scalar matrices by $F$.

\begin{proposition}\label{ut2r}
Let $f\in P_{m,l}$. Then the image of $f$ on $(UT_{2},r)$ is $\{0\}, J, F, \mathcal{K}, \mathcal{S},$ or $\mathcal{K}+J$.
\end{proposition}

\begin{proof}
Let $f=f(y_{1},\dots,y_{l},z_{l+1},\dots,z_{m}) \in P_{m,l}$. By Proposition \ref{pgeneratorsreflection}, we may write $f$ modulo $Id(UT_{2},r)$ as
\begin{eqnarray}\nonumber
f=\alpha y_{1}\cdots y_{l}z_{l+1}\cdots z_{m} +\sum_{i=1}^{l}\alpha_{i}y_{1}\cdots \widehat{y_{i}}\cdots y_{l}z_{l+1}\cdots z_{m-1}[z_{m},y_{i}].
\end{eqnarray}

Notice that if $\alpha=0$, then $f(UT_{2})=J$. Indeed, this follows from the fact that the image of  $[z_{m},y_{i}]$ is contained in $J$ along with the later being an ideal of $UT_{2}$.

We may assume from now on that $\alpha\neq0$ and let us denote $\eta=m-l$.  Note that under the evaluations (\ref{evaluation1}) and by Lemma \ref{lemma12entry} we have that the entry $(1,2)$ of $y_{1}\cdots y_{l}z_{l+1}\cdots z_{m}$ is given by
\[
(y_{1}\cdots y_{l}z_{l+1}\cdots z_{m})_{12}=(-1)^{\eta}w_{1}^{(l+1)}\cdots w_{1}^{(m)}\sum_{i=1}^{l}w_{1}^{(1)}\cdots \widehat{w_{1}^{(i)}}\cdots w_{1}^{(l)}w_{2}^{(i)}.
\]

Moreover, for each $i\in\{1,\dots,l\}$, we have 
$[z_{m},y_{i}]=2w_{1}^{(m)}w_{2}^{(i)}e_{12}$ and also 

\[
(y_{1}\cdots \widehat{y_{i}}\cdots y_{l}z_{l+1}\cdots z_{m-1})_{11}=w_{1}^{(1)}\cdots \widehat{w_{1}^{(i)}}\cdots w_{1}^{(l)}w_{1}^{(l+1)}\cdots w_{1}^{(m-1)}.
\]
Thus
\[
y_{1}\cdots \widehat{y_{1}}\cdots y_{l}z_{l+1}\cdots z_{m-1}[z_{m},y_{i}]=2w_{1}^{(1)}\cdots \widehat{w_{1}^{(i)}}\cdots w_{1}^{(m)}w_{2}^{(i)}e_{12}.
\]
Therefore we conclude that $f(y_{1},\dots,y_{l},z_{l+1},\dots,z_{m})$ is given by the following sum:
\[
\alpha w_{1}^{(1)}\cdots w_{1}^{(m)}(e_{11}+(-1)^{\eta}e_{22}) + \sum_{i=1}^{l}((-1)^{\eta}\alpha+2\alpha_{i})w_{1}^{(1)}\cdots \widehat{w_{1}^{(i)}}\cdots w_{1}^{(m)}w_{2}^{(i)}e_{12}.
\]

If $(-1)^{\eta}\alpha+2\alpha_{i}=0$ for all $i\in\{1,\dots,l\}$, then one can see that $f(UT_2)=F$ or $ f(UT_2)=\mathcal{K}$, according to $\eta$ being even or odd, respectively. 

We assume now that $(-1)^{\eta}\alpha+2\alpha_{k}\neq0$ for some $k\in\{1,\dots,l\}$. We thus have that $f(UT_2)\subset \mathcal{S}$, in case $\eta$ is even, or $f(UT_2)\subset \mathcal{K}+J$ otherwise. We further perform the following evaluation on the commutative variables $w_{j}^{(i)}$:
\begin{itemize}
\item $w_{1}^{(i)}=1$ for $i=l+1,\dots,m$;
\item $w_{2}^{(i)}=0$ for all $i\neq k$;
\item $w_{1}^{(i)}=1$ for all $i\in\{1,\dots,l\}\setminus\{k\}$.
\end{itemize}
Therefore, 
\[
f(y_{1},\dots,y_{l},z_{l+1},\dots,z_{m})=\alpha w_{1}^{(k)}(e_{11}+(-1)^{\eta}e_{22})+((-1)^{\eta}\alpha+2\alpha_{k})w_{2}^{(k)}e_{12}
\]
which implies $f(UT_2)=\mathcal{S}$ or $f(UT_2)=\mathcal{K}+J$.
\end{proof}

\subsubsection{Symplectic involution}

We now turn our attention to the sympletic involution on $UT_{2}$.

It is easy to check that the spaces of symmetric and skew-symmetric elements are given respectively as 
\[\mathcal{S}=\spa\{e_{11}+e_{22}\} \ \mbox{and} \ \mathcal{K}=\spa\{e_{11}-e_{22},e_{12}\}. 
\]

Once again we recall the corresponding result from \cite{VincenzoKoshlukovScala} (see also \cite{UrureGoncalves}).

\begin{proposition}\label{pgeneratorssympletic}
Let $F$ be a field of characteristic different from $2$. Then the vector space $F\langle Y\cup Z \rangle$ is generated, modulo $Id(UT_{2},s)$, by the polynomials
\[
y^{p_{1}}\cdots y_{n}^{p_{n}}[z_{j},z_{i}]z_{i}^{q_{i}}z_{i+1}^{q_{i+1}}\cdots z_{m}^{q_{m}} \ \mbox{and} \ y_{1}^{p_{1}}\cdots y_{n}^{p_{n}}z_{1}^{q_{1}}\cdots z_{m}^{q_{m}}
\]
where $n\geq1, m\geq1,$ $p_1,\dots,p_n,q_1,\dots,q_m\geq 0$ and $j>i$.
\end{proposition}

We consider evaluations of the variables $z_{i}$'s by matrices with entries in $F[W]$:  
\begin{eqnarray}\label{evaluation2}
z_{i}=\begin{pmatrix}
     w_{1}^{(i)}&w_{2}^{(i)}  \\
     & -w_{1}^{(i)}
\end{pmatrix}
\end{eqnarray}
for $i\in\{1,\dots,m\}$. 

\begin{lemma}
In light of \eqref{evaluation2}, the matrix $z_{1}\cdots z_{m}$ is given by 
\[
w_{1}^{(1)}\cdots w_{1}^{(m)}(e_{11}+(-1)^{m}e_{22})+\sum_{i=1}^{m}(-1)^{i+m}w_{1}^{(1)}\cdots \widehat{w_{1}^{(i)}}\cdots w_{1}^{(m)}w_{2}^{(i)}e_{12}
\]
\end{lemma}

\begin{proof}
Induction on $m$. 
\end{proof}

The proof of the next proposition follows the same ideas from the reflexive case. 

\begin{proposition}\label{ut2s}
Let $f\in P_{m,l}$. Then the image of $f$ on $(UT_{2},s)$ is $\{0\},J,\mathcal{S},\mathcal{K},\mathcal{K}\cap \mathcal{D}, \mathcal{S}+J$, where $\mathcal{D}$ denotes the space of diagonal matrices.
\end{proposition}

\begin{proof}

By Proposition \ref{pgeneratorssympletic} and since $\mathcal{S}=F$, we may consider  
\[
f(z_{1},\dots,z_{m})=\alpha_{0}z_{1}\cdots z_{m}+\sum_{i=1}^{m}\alpha_{i}[z_{i},z_{1}]z_{2}\cdots \widehat{z_{i}}\cdots z_{m}.
\]
Note that 
\[
[z_{i},z_{1}]=2(w_{1}^{(i)}w_{2}^{(1)}-w_{2}^{(i)}w_{1}^{(1)})e_{12}
\]
and 
\[
(z_{2}\cdots \widehat{z_{i}}\cdots z_{m})_{22}=(-1)^{m}w_{1}^{(2)}\cdots \widehat{w_{1}^{(i)}}\cdots w_{1}^{(m)}.
\]
Hence,
\[
[z_{i},z_{1}]z_{2}\cdots \widehat{z_{i}}\cdots z_{m}=(-1)^{m}2(w_{1}^{(2)}\cdots w_{1}^{(m)}w_{2}^{(1)}-w_{1}^{(1)}\cdots \widehat{w_{1}^{(i)}}\cdots w_{1}^{(m)}w_{2}^{(i)})e_{12}
\]
and therefore the evaluation of $f$ on the matrices $z_1,\dots,z_m$ is given by the sum of the following two matrices
\[
\alpha_{0}w_{1}^{(1)}\cdots w_{1}^{(m)}(e_{11}+(-1)^{m}e_{22})
\]
and
\begin{eqnarray}\nonumber 
\bigg( (-1)^{m}(-\alpha_{0}+\sum_{i=1}^{m}2\alpha_{i})w_{1}^{(2)}\cdots w_{1}^{(m)}w_{2}^{(1)}\\\nonumber 
+\sum_{i=2}^{m}(-1)^{m}((-1)^{i}\alpha_{0}+2\alpha_{i})w_{1}^{(1)}\cdots \widehat{w_{1}^{(i)}}\cdots w_{1}^{(m)}w_{2}^{(i)}\bigg)e_{12}
\end{eqnarray}

We may assume $\alpha_0\neq0$, otherwise the image would be already determined since $f(UT_2)\subset J$.

If $-\alpha_{0}+\sum_{i=1}^{m}2\alpha_{i}=0$ and $(-1)^{i}\alpha_{0}+2\alpha_{i}=0$ for all $i\in\{2,\dots,m\}$, then it is easy to see that $f(UT_2)=\spa\{e_{11}+(-1)^{n}e_{22}\}\in\{\mathcal{S},\mathcal{K}\cap \mathcal{D}\}$.
Otherwise let us first assume that $-\alpha_{0}+\sum_{i=1}^{m}2\alpha_{i}\neq0$. Then we perform the following evaluation of the commutatite variables $w_j^{(i)}$: 
\begin{itemize}
\item $w_{2}^{(i)}=0$ for all $i\in\{2,\dots,m\}$;
\item $w_{1}^{(i)}=1$ for all $i\in\{2,\dots,m\}$.
\end{itemize}
We therefore have  
\[
f(z_1,\dots,z_m)=\alpha_{0}w_{1}^{(1)}(e_{11}+(-1)^{m}e_{22})+
(-1)^{m}(-\alpha_{0}+\sum_{i=1}^{m}2\alpha_{i})w_{2}^{(1)}e_{12}
\]
which implies that $f(UT_2)=\spa\{e_{11}+(-1)^{m}e_{22},e_{12}\}\in\{\mathcal{K},\mathcal{S}+J\}$.

Assume now that $(-1)^{i}\alpha_0+2\alpha_i\neq0$ for some $i\in\{2,\dots,m\}$. We thus set the evaluation
\begin{itemize}
\item $w_{2}^{(j)}=0$ for all $j\in\{1,\dots,m\}\setminus\{i\}$;
\item $w_{1}^{(j)}=1$ for all $j\in\{1,\dots,m\}\setminus\{i\}$.
\end{itemize}
We obtain then 
\[
f(z_1,\dots,z_m)=\alpha_0 w_1^{(i)}(e_{11}+(-1)^{m}e_{22})+(-1)^{m}((-1)^{i}+2\alpha_i)w_{2}^{(i)}e_{12},
\]
which clearly leads us to $f(UT_2)=\spa\{e_{11}+(-1)^{m}e_{22},e_{12}\}\in\{\mathcal{K},\mathcal{S}+J\}$.
\end{proof}

\subsection{Grading \boldmath{$\Gamma_{2,2}$}}

Let us first write 
\[
    UT_{2}=\mathcal{A}_{1}\oplus\mathcal{A}_{g},
\]
where $\mathcal{A}_{1}=\spa\{e_{11},e_{22}\}$ and $\mathcal{A}_{g}=\spa\{e_{12}\}$. We further write
\begin{eqnarray}\label{decompositionUT2}
\mathcal{A}_{1}=\mathcal{S}_{1}\oplus\mathcal{K}_{1} \ \mbox{and} \ \mathcal{A}_{g}=\mathcal{S}_{g}\oplus\mathcal{K}_{g}
\end{eqnarray}
with respect to the involution $*\in\{r,s\}$. 

Applying Corollary \ref{corollarydimension2}  yields that the image of multilinear graded polynomial with involution $f$ on $UT_{2}$ is a vector space (the subspaces $V_{i}$ are defined as the ones appearing in the decompositions in \eqref{decompositionUT2}, in accordance with the homogeneous degree and symmetry of the variables occurring in $f$). Indeed, since $\supp(\Gamma_{2,2})$ is abelian, then $f(UT_{2})$ is contained in some homogeneous component, and now one just has to notice that both $\mathcal{A}_{1}$ and $\mathcal{A}_{g}$ have dimension $\leq 2$. 

It is also straightforward to obtain a precise classification of the images of multilinear graded $*$-polynomials on $UT_{2}$ with the grading $\Gamma_{2,2}$.

\begin{proposition}\label{ut2gradedinvolution}
Let $f\in P_{m,l}^G$ and consider the $G$-graded involution $*\in\{r,s\}$ on $UT_2$ with respect to the grading $\Gamma_{2,2}$. Then $f(UT_2)$ is either $\{0\}$ or some (skew-)symmetric part from some homogeneous component.
\end{proposition} 

\begin{proof}
Let us consider $*=r$, since the sympletic case is quite similar. First we note that 
\[
\mathcal{S}_{1}=\spa\{e_{11}+e_{22}\}, \mathcal{K}_{1}= \spa\{e_{11}-e_{22}\}, \mathcal{S}_{g}=\mathcal{A}_{g} \ \mbox{and} \ \mathcal{K}_{g}=\{0\}.
\]

Since $\supp(\Gamma_{2,2})=\{1,g\}$, we may consider that only variables of homogeneous degree $1$ and $g$ occur in $f$. Now note that $\mathcal{S}_g$ is a nilpotent ideal of $UT_2$ of index $2$, and it is one-dimensional as a vector space. Hence if $f$ has at least one variable of homogeneous degree $g$, then $f(UT_2)$ is either $\{0\}$ or $\mathcal{S}_g$. 

So we may assume now that $f$ has only neutral variables.  Since $UT_2$ satisfies the graded $*$-identities $[z_{1,1},z_{2,1}]=0$ and $[y_{1,1},x]=0$, where $x\in Y_G\cup Z_G$, then modulo these identities we may write $f$ as
\[
f=\alpha y_{1,1}\cdots y_{l,1}z_{l+1,1}\cdots z_{m,1},
\]
for some $\alpha\in F$. This will therefore imply that $f(UT_2)$ equals to $\{0\}, \mathcal{S}_1$ or $\mathcal{K}_1$.
\end{proof}

Finally we give the proof of the main theorem of this section.

\begin{proof}[Proof of Theorem \ref{ut2}]
It is enough to apply Theorem \ref{classificationgradedinvolution} and Proposition \ref{epimorphism}, along with Propositions \ref{ut2r},\ref{ut2s},\ref{ut2gradedinvolution}.
\end{proof}
\section{\bf $3\times 3$ matrices}

Let us recall that, up to equivalence, there is only one involution on $UT_{3}$, the reflexive one, given by:
\[
\left(\begin{array}{ccc}
    a_{11} & a_{12} & a_{13}\\
         & a_{22} & a_{23}\\
         &      & a_{33}
\end{array}\right)^{r}=\begin{pmatrix}
    a_{33} & a_{23} & a_{13}\\
         & a_{22} & a_{12}\\
         &      & a_{11}
\end{pmatrix}.
\]

Therefore, the only graded involutions on $UT_3$ are the ones given by the reflexive involution $r$ and gradings:
\begin{itemize}
\item $\Gamma_{1,3}$: the trivial grading;
\item $\Gamma_{2,3}$: $UT_3=\mathcal{A}_{1}\oplus \mathcal{A}_{g}$, where $\mathcal{A}_{1}=\spa\{e_{11},e_{22},e_{33},e_{13}\}$ and $\mathcal{A}_{g}=\spa\{e_{12},e_{23}\}$;
\item $\Gamma_{3,3}$: $UT_{3}=\mathcal{A}_{1}\oplus \mathcal{A}_{g} \oplus \mathcal{A}_{h}$, where $\mathcal{A}_{1}=\spa\{e_{11},e_{22},e_{33}\}$, $\mathcal{A}_{g}=\spa\{e_{12},e_{23}\}$ and $\mathcal{A}_{h}=\spa\{e_{13}\}$.
\end{itemize}

Now we present an example which shows that an analogue of Theorem \ref{ut2} cannot be expected for upper triangular matrices of order $n\geq 3$ with the trivial grading. The polynomial from our example will be in skew-symmetric variables. To this end, we recall that the skew-symmetric part of $UT_n$ with the reflexive involution is given by 
\[
\mathcal{K}=\spa\{e_{ij}-e_{n+1-j,n+1-i}|i\leq j, i,j=1,\dots,n\}.
\]

\begin{proposition}\label{ut3example}
Let $n\geq3$ and let $UT_{n}$ be endowed with the reflexive involution. Then the image of the multilinear polynomial $f(z_1,z_2)=z_1 z_2$ on $UT_{n}$ is not a vector space. 
\end{proposition}

\begin{proof}
Let $n$ be an odd integer (the even case can be treated with minor adjustments), and let us assume that $f(UT_n)$ is a vector space.

Denoting $n_{0}=\frac{n+1}{2}$ we see that 
\begin{align*}
e_{11}+e_{nn}=& \ f(e_{11}-e_{nn},e_{11}-e_{nn})\\
e_{1n}=&f(e_{1,n_{0}}-e_{n_{0},n},-e_{1,n_{0}}+e_{n_{0},n}).
\end{align*}

Hence we must have $e_{11}+e_{nn}+e_{1n}\in f(UT_{n})$, that is, there exist $A,B\in \mathcal{K}$ such that 
\begin{equation}\label{equation for AB}
e_{11}+e_{nn}+e_{1n}=AB.
\end{equation}
Let us write 
\[
A=\sum_{1\leq i\leq j\leq n}a_{ij}e_{ij} \ \mbox{and} \ B=\sum_{1\leq i\leq j\leq n}b_{ij}e_{ij},
\]
and since $A,B\in\mathcal{K}$, we additionally have that $a_{ij}=-a_{n+1-j,n+1-i}$ and $b_{ij}=-b_{n+1-j,n+1-i}$, for all $i,j$.

We claim that $(AB)_{1n}=0$, which clearly leads us to a contradiction. To prove the claim, let us compute the following entries of the product $AB$:

\begin{enumerate}
    \item[1.] ``Half first row'': for $j=2,\dots,n_{0}$, the entry $(1,j)$ is given by 
    \[
    \sum_{i=1}^{j}a_{1i}b_{ij}
    \]
    \item[2.] ``Half last column'': for $i=n_{0},\dots,n-1$, the entry $(i,n)$ is given by
    \[
    \sum_{j=i}^{n}a_{ij}b_{jn}.
    \]
\end{enumerate}

We rewrite the entry $(i,n)$ above as 
\[
\sum_{j=i}^{n}a_{ij}b_{jn}=\sum_{j=i}^{n}a_{n+1-j,n+1-i}b_{1,n+1-j}=\sum_{k=1}^{l}a_{kl}b_{1k},
\]
for $l=2,\dots,n_{0}$.

We now prove that $a_{1j}=b_{1j}=0$ for $j=2,\dots,n_{0}$. To this end we proceed by induction on $n_{0}$. For the base of the induction, we notice that the entries $(1,2)$ and $(n-1,n)$ along with Equation \eqref{equation for AB} give us
\[
a_{11}b_{12}+a_{12}b_{22}=0 \ \mbox{and} \ a_{22}b_{12}+a_{12}b_{11}=0
\]
Since $a_{11}b_{11}\neq0$ and $a_{22}b_{22}=0$, we therefore get $a_{12}=b_{12}=0$. We assume now $a_{1j}=b_{1j}=0$ for $j<n_{0}$. Considering the entries $(1,n_{0})$ and $(n_{0},n)$ along with Equation \eqref{equation for AB} we have
\[
\sum_{i=1}^{n_{0}}a_{1i}b_{i,n_{0}}=0 \ \mbox{and} \ \sum_{k=1}^{n_{0}}a_{k,n_{0}}b_{1k}=0.
\]
By our induction hypothesis, the two equations above reduce to
\[
a_{11}b_{1,n_{0}}+a_{1,n_{0}}b_{n_{0},n_{0}}=0 \ \mbox{and} \ a_{1,n_{0}}b_{11}+a_{n_{0},n_{0}}b_{1,n_{0}}=0
\]
Now it is enough to use that $a_{11}b_{11}\neq0$ and $a_{n_{0},n_{0}}b_{n_{0},n_{0}}=0$ in order to get $a_{1,n_{0}}=b_{1,n_{0}}=0$.

Finally, to obtain our claim we just need to notice that
\[
(AB)_{1n}=\sum_{i=1}^{n}a_{1i}b_{in}=-\sum_{i=1}^{n}a_{1i}b_{1,n+1-i}
\]
and that $a_{1n}=b_{1n}=a_{1i}=b_{1i}=0$ for $i=2,\dots,n_{0}$.
\end{proof}

On the other hand, we obtain a different picture when we consider non-trivial gradings on $UT_{3}$. 

\begin{theorem}\label{ut3}
    Let $F$ be an algebraically closed field of characteristic different from $2$. Let $f\in P_{m,l}^G$. Assume that $UT_{3}$ is endowed with some $G$-graded involution $*$ on a non-trivial grading $\Gamma$ such that $supp(\Gamma)$ generates $G$. Then, the image of $f$ on $UT_{3}$ is a homogeneous vector space.
\end{theorem}

As in the $2\times 2$ case, we divide the proof of Theorem \ref{ut3} in two sections.

\subsection{The grading \boldmath{$\Gamma_{2,3}$}}

We recall that $\mathcal{A}_1=\spa\{e_{11},e_{22},e_{33},e_{13}\}$ and $\mathcal{A}_g=\spa\{e_{12},e_{23}\}$. Hence,
\[
\mathcal{S}_1=\spa\{e_{11}+e_{33},e_{22},e_{13}\}, \mathcal{S}_g=\spa\{e_{12}+e_{23}\},
\]
\[
\mathcal{K}_1=\spa\{e_{11}-e_{33}\} \ \mbox{and} \ \mathcal{K}_g=\spa\{e_{12}+e_{23}\}.
\]

Since the homogeneous components are invariant under the involution, we may regard the neutral component $\mathcal{A}_{1}$ as an algebra with involution as well. We notice some similarities between the reflexive case on $UT_{2}$ and $\mathcal{A}_{1}$.

\begin{lemma}
The neutral component of $UT_{3}$ with reflexive involution and grading $\Gamma_{2,3}$ satisfies the following identities:
\begin{eqnarray*}
\begin{array}{ccccc}
     (i) \ [y_{1},y_{2}]; & (ii) \ [z_{1},z_{2}]; & (iii) \ [y_{1},z_{1}][y_{2},z_{2}]; & (iv) \ z_{1}y_{1}z_{2}-z_{2}y_{1}z_{1}.
\end{array}
\end{eqnarray*}
\end{lemma}

\begin{proof}
It is immediate, hence omitted.
\end{proof}

One can see that the identities from the lemma above are exactly the ones satisfied by $UT_{2}$ with the reflexive involution (see \cite[Theorem 3.1]{VincenzoKoshlukovScala}). So it is reasonable to expect that we can use the same approach from Proposition \ref{ut2r}. Let us thus consider the following evaluation of the variables $y$'s and $z$'s by matrices over $F[W]$:
\begin{eqnarray}\label{matrices3}
y_i=\begin{pmatrix}
w_1^{(i)} & 0 & w_3^{(i)}\\
 & w_2^{(i)} & 0\\
 & & w_1^{(i)} 
\end{pmatrix} \ \mbox{and} \ z_j=\begin{pmatrix}
w_1^{(j)} & 0 & 0\\
 & 0 & 0\\
 & & -w_1^{(j)}
\end{pmatrix}
\end{eqnarray}

\begin{lemma}\label{ut3evaluation}
Let $y_{i}, i=1,\dots,m$, as in equation \eqref{matrices3}. Then the entry $(1,3)$ of $y_{1}\cdots y_{m}$ is given by 
\[
\sum_{i=1}^{m}w_{1}^{(1)}\cdots \widehat{w_{1}^{(i)}}\cdots w_{1}^{(m)}w_{3}^{(i)}.
\]
\end{lemma}

\begin{proof}
    Induction on $m$.
\end{proof}

The proof of the next lemma follows the ideas from Proposition \ref{ut2r}. 

\begin{lemma}\label{supp2ut3}
    Let $f\in P_{m,l}$. Then $f(\mathcal{A}_{1})$ is $\{0\}, J, \mathcal{S}_{1}\cap \mathcal{D}, (\mathcal{S}_1\cap\mathcal{D})+J, \mathcal{S}_{1}, \mathcal{K}_{1} $ or $\mathcal{K}_{1}+J$, where $J$ denotes the subspace spanned by $\{e_{13}\}$, and $\mathcal{D}$ stands for the linear span of $\{e_{11},e_{33}\}$.
\end{lemma}

\begin{proof}
    By Proposition \ref{pgeneratorsreflection} we may write $f$ as
    \[
    f=\alpha y_{1}\cdots y_{l}z_{l+1}\cdots z_{m}+\sum_{i=1}^{l}\alpha_{i}y_{1}\cdots \widehat{y_{i}}\cdots y_{l}z_{l+1}\cdots z_{m-1}[z_{m},y_{i}].
    \]
    
    We will assume $\alpha\neq0$, since otherwise we clearly have $f(UT_{3})\in\{\{0\},J\}$.

    Let $\eta=m-l$. Then notice that the main diagonal of $f$ is given by 
    \[
    w_{1}^{(1)}\cdots w_{1}^{(m)}(e_{11}+(-1)^{\eta}e_{33})+\lambda w_{2}^{(1)}\cdots w_{2}^{(m)}e_{22},
    \]
    where $\lambda=0$ if there is no skew-symmetric variable in $f$, and $\lambda=1$ otherwise. The case where $f$ has only symmetric variables is obvious, indeed one can easily obtain $f(UT_{3})=\mathcal{S}_{1}$. Hence we may assume that $f$ has skew-symmetric variables and also $\lambda=0$. Hence 
\[
f(UT_{3})\subset \mathcal ({S}_{1}\cap \mathcal{D})+J \ \mbox{or} \ f(UT_{3})\subset \mathcal{K}_{1}+J,
\]
accordingly $\eta$ is even or $\eta$ is odd, respectively. Now let us turn our attention to the nilpotent part of $f$.  First we notice that $[z_{m},y_{i}]=2w_{3}^{(i)}w_{1}^{(m)}e_{13}$, and therefore
    \[
    y_{1}\cdots \widehat{y_{i}}\cdots y_{l}z_{l+1}\cdots z_{m-1}[z_{m},y_{i}]=2w_{1}^{(1)}\cdots \widehat{w_{1}^{(i)}}\cdots w_{1}^{(m-1)}w_{3}^{(i)}w_{1}^{(m)}.
    \]
    
    Hence in light of Lemma \ref{ut3evaluation} we have that $f(y_1,\dots,y_l,z_{l+1},\dots,z_m)$ equals 
    \[
    \alpha w_{1}^{(1)}\cdots w_{1}^{(m)}(e_{11}+(-1)^{\eta}e_{33}) +  \sum_{i=1}^{l}((-1)^{\eta}\alpha +2\alpha_{i})w_{1}^{(1)}\cdots \widehat{w_{1}^{(i)}} \cdots w_{1}^{(m)}w_{3}^{(i)}e_{13}.
    \]
    
    Proceding in a silimar fashion as in Proposition \ref{ut2r}, one can see that 
\[
f(UT_3)\in\{\mathcal{K}_1,\mathcal{K}_1+J,\mathcal{S}_1\cap \mathcal{D},(\mathcal{S}_1\cap \mathcal{D})+J\},
\]
which finishes the proof.
\end{proof}

\begin{proposition}\label{propositionZ2}
Let $f\in P_{m,l}^G$. Then $f(UT_{3})$ is $\{0\}, J, \mathcal{S}_{1}\cap \mathcal{D}, (\mathcal{S}_1\cap\mathcal{D})+J, \mathcal{S}_{1}, \mathcal{K}_{1}, \mathcal{K}_{1}+J$, some one-dimensional subspace of $\mathcal{A}_{1}$ or $\mathcal{A}_{1}$. Moreover, every one-dimensional subspace of $\mathcal{A}_{1}$ can be realized as the image of some multilinear graded polynomial with involution on $UT_{3}$, as well as the homogeneous component $\mathcal{A}_{1}$. 
\end{proposition}

\begin{proof}
The proof of the first part is clear from Lemma \ref{supp2ut3} and Corollary \ref{corollarydimension2}. Now, fix $\alpha,\beta\in F$ and consider the one-dimensional subspace 
\[
\mathcal{V}=\spa\{\alpha e_{12}+\beta e_{23}\}\subset \mathcal{A}_{1}.
\] 

It is straightforward to  check that the image of $\alpha z_{1,0}z_{1,1}+\beta z_{1,1}z_{1,0}$ on $UT_{3}$ is exactly $\mathcal{V}$, and that the image of $y_{1,0}y_{1,1}$ on $UT_{3}$ is exactly $\mathcal{A}_{1}$.
\end{proof}

\subsection{The grading \boldmath{$\Gamma_{3,3}$}}

We recall that $\mathcal{A}_{1}=\spa\{e_{11},e_{22},e_{33}\}, \mathcal{A}_{g}=\spa\{e_{12},e_{23}\}$ and $\mathcal{A}_{h}=\spa\{e_{13}\}$. We therefore write 
\[
UT_{3}=\mathcal{S}_{1}\oplus \mathcal{K}_{1}\oplus \mathcal{S}_{g}\oplus \mathcal{K}_{g}\oplus \mathcal{S}_{h}\oplus \mathcal{K}_{h}
\]
where $\mathcal{S}_{1}=\spa\{e_{11}+e_{33},e_{22}\},\mathcal{K}_{1}=\spa\{e_{11}-e_{33}\},\mathcal{S}_{g}=\spa\{e_{12}+e_{23}\},\mathcal{K}_{g}=\spa\{e_{12}-e_{23}\},\mathcal{S}_{h}=\spa\{e_{13}\}$ and $\mathcal{K}_{h}=\{0\}.$
Hence we have the following result.

\begin{proposition}\label{propositionZ3}
    Let $F$ be a field of characteristic different from $2$ and let $f\in P_{m,l}^G$. Then $f(UT_{3})$ is $\{0\},\mathcal{S}_{1},\mathcal{K}_{1},(\mathcal{K}_{1})^{2}$, some one-dimensional subspace of $\mathcal{A}_{g}$, $\mathcal{A}_{g}$ or $\mathcal{A}_{h}$. Moreover, any subspace of $\mathcal{A}_{g}$ can be realized as the image of some multilinear polynomial on $UT_{3}$.
\end{proposition}

\begin{proof}
Assume first that $f$ has (skew-)symmetric neutral variables only. Since $UT_3$ satisfies $[x_1,x_2]$ where $x_1,x_2$ are any (skew-)symmetric neutral variables, then  the image of $f$ on $UT_3$ is either zero or the image of a word in the form $y_{1,1}\cdots y_{l,1} z_{l+1,1}\cdots z_{m,1}$ on $UT_3$. Clearly the former gives us $f(UT_3)\in\{\mathcal{S}_{1},\mathcal{K}_{1},(\mathcal{K}_{1})^{2}\}$.

Let us assume now the $f$ has non-neutral variables. Since $x_1 x_2$ is an identity for $UT_3$ when $x_1$ has homogeneous degree $g$ and $x_{2}$ has homogeneous degree $h$, then $f(UT_3)$ is always contained in some non-neutral homogeneous component. In this case, we apply Corollary \ref{corollarydimension2} to conclude that $f(UT_{3})$ is a vector subspace of $\mathcal{A}_{g}$ or $\mathcal{A}_{h}$. Moreover, we notice that $\mathcal{A}_{g}=f(UT_{3})$ for $f=y_{1,0}y_{1,1}$, and for arbitrarily fixed $\alpha,\beta\in F$ and
\[
\mathcal{V}=\spa\{\alpha e_{12}+\beta e_{23}\}
\]
we have $\mathcal{V}=f(UT_{3})$ for $f=\alpha z_{1,0}y_{1,1}-\beta y_{1,1}z_{1,0}$.
\end{proof}

\begin{proof}[Proof of Theorem \ref{ut3}]
It is enough to apply Theorem \ref{classificationgradedinvolution} and Proposition \ref{epimorphism}, along with Propositions \ref{propositionZ2} and \ref{propositionZ3}.
\end{proof}

We finish this section by showing that, in general, Theorem \ref{ut3} can not be expected to hold for $UT_{n}$, $n\geq 4$, not even for the canonical $\Z_{n}$-grading. 

\begin{proposition}
    Let $n\geq4$ and let $UT_{n}$ be endowed with the canonical $\mathbb{Z}_{n}$-grading and reflexive involution. Then the image of the multilinear polynomial $f(y_{1,\overline{0}},y_{2,\overline{1}})=y_{1,\overline{0}}y_{2,\overline{1}}$ on $UT_{n}$ is not a vector space.
    \end{proposition}

    \begin{proof}
        Let us suppose $n$ even (the odd case is analogous). We recall that the symmetric parts of homogeneous degree $\overline{0}$ and $\overline{1}$ are given by
        \begin{align*}
        \mathcal{S}_{\overline{0}}=& \ \spa\{e_{ii}+e_{n+1-i,n+1-i}; i=1,\dots,n/2\}, \mbox{and}\\
        \mathcal{S}_{\overline{1}}=& \ \spa\{e_{n/2,(n+2)/2}, e_{i,i+1}+e_{n-i,n+1-i};i=1,\dots,-1+n/2\}.
        \end{align*}
        
        Assuming that $f(UT_{n})$ is a vector space, and noticing that 
        \begin{align*}
            f(e_{11}+e_{nn},e_{12}+e_{n-1,n})=& \ e_{12}\\
            f(e_{22}+e_{n-1,n-1},e_{23}+e_{n-2,n-1})=& \ e_{23}
        \end{align*}
        we therefore have $e_{12}+e_{23}\in f(UT_{n})$ (we take $e_{23}=f(e_{22}+e_{33},e_{23})$ in case $n=4$).
        Hence there exist $A\in \mathcal{S}_{0}$ and $B\in \mathcal{S}_{1}$ such that $e_{12}+e_{23}=AB$. 
        
        Writing $A=\sum_{i=1}^{n}a_{ii}e_{ii}$ and $B=\sum_{i=1}^{n-1}b_{i,i+1}e_{i,i+1}$, we therefore have
        \begin{align*}
            (AB)_{12}=& \ a_{11}b_{12}\\
            (AB)_{23}=& \ a_{22}b_{23}\\
            (AB)_{n-1,n}=& \ a_{n-1,n-1}b_{n-1,n}=a_{22}b_{12}
        \end{align*}
        
        Since the entry $(n-1,n)$ of $AB$ is zero, we conclude that either $(AB)_{12}=0$ or $(AB)_{23}=0$, a contradiction.
    \end{proof}

\section{Acknowlegments}

The author is thankful to his supervisors Plamen Koshlukov and Matej Bre\v{s}ar for their support.

\end{document}